\documentclass[12pt]{amsart}

\usepackage{amsmath,amssymb,array,latexsym,amsthm,pdfsync,hyperref}
\theoremstyle{definition}

\newtheorem{thm}{Theorem}
\newtheorem{lem}{Lemma}
\newtheorem*{thma}{Theorem A}
\newtheorem*{lema}{Lemma A}

\newtheorem{prop}{Proposition}

\newtheorem{defin}{Definition}
\theoremstyle{remark}
\theoremstyle{remark}\newtheorem{rem}{Remark}
\theoremstyle{remark}\newtheorem{notat}{Notation}

\def\lra{\longrightarrow}

\def\Ex{\mathbb{E}}

\def\notin{\epsilon \hspace{-0.37em}/}

\newcommand{\E}{\ensuremath{\mathbb E}}
\newtheorem{kz thms}{Theorem}

\begin{document}

\title{Empirical Quantile CLTs for Some Self-Similar Processes}
\author{James Kuelbs} 
\address{James Kuelbs\\Department of Mathematics,  University of Wisconsin, Madison, WI 53706-1388}
\email{kuelbs@math.wisc.edu}
\author{Joel Zinn}
\address{\noindent Joel Zinn\\Department of Mathematics, Texas A\&M University, College Station, TX 77843-3368}
\email{jzinn@math.tamu.edu}
\begin{abstract}  In \cite{swanson-scaled-median} a CLT for the sample median of independent Brownian motions with value $0$ at $0$  was proved. Here we extend this result in two ways. We prove such a result for a collection of self-similar processes which include the fractional Brownian motions and also all stationary, independent increment symmetric stable processes tied down at $0$. Second, our results hold uniformly over all quantiles in a compact sub-interval of $(0,1)$. We also examine sample function properties connected with these CLTs.

\end{abstract}

\maketitle
\section{Introduction}\label{sec1}

Let  $X=\{X(t)\colon\ t \in E\}\equiv \{X_t :t \in E\}$ be a stochastic process with $P(X(\cdot) \in D(E))=1,$ where $E$ is a set and $D(E)$ is a collection of real valued 
functions on $E$. Also, let $\mathcal{C}=\{C_{s,x}\colon\  s \in E, x \in \mathbb{R}\},$ where $C_{s,x}=\{z \in D(E)\colon\  z(s) \le x\}, s \in E, x \in \mathbb{R}$. If  $\{X_{j}\}_{j=1}^{\infty}$ are i.i.d. copies of the stochastic process $X$ and $F(t,x):=P(X(t) \le x)=P( X(\cdot) \in C_{t,x})$, then the empirical distributions built on $\mathcal{C}$ (or built on the process $X$) are defined by 
\[
F_n(t,x)= \frac{1}{n} \sum_{i=1}^n I_{(-\infty, x]}(X_i(t))= \frac{1}{n}\sum_{i=1}^n I_{\{X_i \in C_{t,x}\}},\ C_{t,x} \in \mathcal{C},
\]
and we say $X$ is the input process.

The empirical processes indexed by $\mathcal{C}$ (or just $E\times\mathbb{R})$ and built from the process, $X$, are given by
\begin{align}\label{nu_{n}}\nu_{n}(t,x):=\sqrt{n}\bigl(F_{n}(t,x)-F(t,x)\bigr).
\end{align}

In \cite{kkz} we studied the central limit theorem in this setting, that is, we found sufficient conditions for a pair $(\mathcal{C}, P)$, where $P$ is the law of $X$ on $D(E)$, ensuring that the sequence of empirical processes $\{\nu_n(t,x): (t,x) \in E\times \mathbb{R}\}, n \ge1,$
converge to a centered Gaussian process, $G=\{G_{t,x}\colon\  (t,x)\in E\times\mathbb{R}\}$ with covariance
\[
\E(G(s,x)G(t,y)) = \E([I(X_s \leq x)-P(X_s \leq x)][I(X_t \leq y) -P(X_t \leq y)]).
\]

In \cite{kz-quant}, using an idea of Vervaat \cite{vervaat-quantiles} and the results of \cite{kkz}, we obtain (central) limit theorems for empirical quantile processes built from  i.i.d. copies $\{X_j:j \ge 1 \}$ of $X$, where $X$ can be chosen from a broad collection of Gaussian processes, compound Poisson processes, symmetric stationary independent increment stable processes, and certain martingales.

The quantiles and empirical quantiles are defined as the left-continuous inverses of $F(t,x)$ and $F_{n}(t,x)$ in the variable $x$, respectively: 
\begin{equation}\label{quantile} \tau_{\alpha}(t)=F^{-1}(t,\alpha)=\inf\{x\colon F(t,x)\ge \alpha\}
\end{equation}
and 
\begin{equation}\label{empirical quantile} \tau^{n}_{\alpha}(t)=F_{n}^{-1}(t,\alpha)=\inf\{x\colon F_{n}(t,x)\ge \alpha\}.
\end{equation}
The CLTs obtained in \cite{kz-quant}, and which we continue to seek, include Gaussian limits for the empirical quantile processes  
\begin{equation}\sqrt{n}\bigl(F_{n}^{-1}(t,\alpha)-F^{-1}(t,\alpha)\bigr)\notag,
\end{equation}
or in the more compact notation for
\begin{equation*}\sqrt{n}\bigl( \tau^{n}_{\alpha}(t)- \tau_{\alpha}(t)\bigr).
\end{equation*}

The papers of Swanson \cite{swanson-scaled-median} and \cite{swanson-fluctuations} dealing with Brownian motion were the first to motivate our interest in this set of problems, but the techniques we use are quite different, and apply to a much broader set of input processes. Nevertheless, the results in \cite{kz-quant} do not fully capture those of Swanson, even for the Brownian motion case, and the main purpose of this paper is to show how these results can be extended to obtain those of Swanson for Brownian motion, and also a large class of self-similar processes. To understand this problem in more depth, we turn to some additional details. 

In the first of these papers Swanson obtained a central limit theorem for the median process, when in our terminology the input process $\{X_t: t \ge 0\}$ is a sample continuous Brownian motion tied down to have value $0$ at time $0$. In the second he establishes a CLT for the empirical quantile process for each fixed $\alpha \in (0,1)$, but now $\{X_t: t \ge 0\}$ is assumed to be a sample continuous Brownian motion whose distribution at time zero is assumed to have a density with a unique $\alpha$ quantile. In particular, his results are uniform in $t \in [0,T]$ for $T \in (0,\infty)$, but only for fixed $\alpha$, and only for empirical quantile processes of Brownian motion.
Hence the empirical quantile CLTs in \cite{kz-quant} are more general in many ways, but because our method of proof for these results depends on the empirical CLT over $\mathcal{C}$ holding, which fails when $X$ is Brownian motion on $[0,T]$ with $X(0)=0$, they apply to Brownian motion on $[0,T]$ only when we start with a nice density at time $0$. On the other hand, as can be seen from Theorem 2 in \cite{kz-quant}, our results apply to general classes of processes, including symmetric stationary independent increment stable processes and fractional Brownian motions, as long as they have a nice density at time $0$. Moreover, these quantile CLTs are uniform in $t \in [0,T]$ and also in $\alpha \in I$, where $I$ is a closed subinterval of $(0,1)$. However, as can be seen from  \cite{kkz} and Corollary 3 in \cite{kz-quant}, many of these processes fail the empirical CLT over $\mathcal{C}$ on $[0,T]$ when they start at $0$ when $t=0$, so to apply our quantile CLTs in \cite{kz-quant} to such processes we must find a way to circumvent the assumption of the empirical CLT holding over $\mathcal{C}$ on $[0,T]$. 

This is our first task, and the precise result appears in the statement of Theorem \ref{scale-quant} of the next section.  In particular, it is shown that a CLT for empirical quantile processes holds for a variety of self-similar input processes with value $0$ at time $t=0$. These weak convergence results are uniform simultaneously in the process parameter $t\in [0,T]$  and with respect to $\alpha$ in a fixed closed subinterval of $(0,1)$.  The main tool to handle the $t$'s in a neighborhood of the origin, is to relate the process in that neighborhood to the process in an interval bounded away from the origin where the results in \cite{kz-quant} can be applied. This leads us to define ``scalable processes'', which controls certain expectations - if they exist. We, then, have to show that while similar expectations do not exist for, e.g., stable processes, they nonetheless exist for quantiles of such processes. 

Theorem \ref{scale-quant-app} follows by an application of Theorem \ref{scale-quant}, and  shows that fractional Brownian motions and symmetric stationary independent increment stable processes have empirical quantile CLTs on $[0,T]$ even if they start at zero when $t=0$. Theorem \ref{quant-CLT-paths} investigates how sample path properties of the input process $\{X(t): t\ge 0\}$ can be used to obtain our CLTs in related function spaces, and a special case of its applications show how the results of  \cite{swanson-scaled-median} and \cite{swanson-fluctuations} for Brownian motion can be obtained. Of course, there are many other self-similar processes, and a natural question is whether similar results hold in a wide variety of other situations. We do not settle that question here, but in section five we indicate a number of important examples to which the results in Theorems \ref{scale-quant} and \ref{quant-CLT-paths} apply.

\section{Definitions, Notation, and Statements of Results}

\begin{defin} A sequence of stochastic processes, $\{W_{n}(t)\colon t\in T\}$, with almost surely bounded sample paths is said to converge (weakly) or in law in $\ell_{\infty}(T)$ to a Gaussian r.v., $G$, with law $\gamma$,  if $\gamma$ is Radon on 
$\ell_{\infty}(T)$ (with the usual sup-norm), or equivalently, (see Example 1.5.10 in \cite{vw}), that $G$ has sample paths which are bounded and uniformly continuous on $T$ with respect to the pseudo-metric
\begin{equation}\label{eq1.1}
d(s,t) 
= \E([G(s)-G(t)]^2)^{\frac{1}{2}}.
\end{equation}
And further, for every bounded, continuous $F\colon\  \ell_{\infty}(T)\lra \mathbb{R}$, 
\[
\lim_{n\to\infty}\E^{*}F(W_{n})=\E F(G),
\]
where $\E^{*}$ denotes the upper expectation (see, e.g., p. 94 in \cite{Dudley-unif-clt}). We denote this by writing $W_n \Rightarrow G$.
\end{defin}

In what follows the stochastic processes can be the empirical processes, the empirical quantile processes or rather general processes (as in Lemma \ref{weak-conv-paths}). 

To rescale small time intervals near zero to intervals where results in \cite{kz-quant} will hold, we assume the the processes $X$ are self-similar, which is given by:

\begin{defin} A stochastic process, $\{X(t): t\ge 0\}$ 
is said to be $H$-self-similar if there exists $H \in (0,\infty)$ such that for all $c > 0$ the processes
$\{X(ct): t\ge 0\}$ and $\{c^H X(t): t\ge 0\}$ have the same finite dimensional distributions.
\end{defin}

\noindent We also will have need for another form of rescaling, which is given by:

\begin{defin} A stochastic process $\{W(t,\alpha):t\geq 0, \alpha\in A\}$ will be called $H$-scalable (in $t$), if $W(0,\alpha)=0$ for all $\alpha\in A$ and for some $H \in (0,\infty)$ and all $c \in (0,\infty)$ the processes  
\[\{W(c t,\alpha):t\geq 0, \alpha\in A\} \text{ \rm and } \{c^{H}W(t,\alpha):t\geq 0, \alpha\in A\}
\] 
have the same finite dimensional distributions.
\end{defin} 

Some notation and additional assumptions on the the stochastic processes to which our results apply are as follows.
\bigskip

\noindent (A-1): The stochastic processes $\{X(t):t \ge 0\}$ and $\{X_j(t):t \ge 0\}, j \ge 1,$ are i.i.d. on a suitable probability space $(\Omega,\mathcal{F},P)$ with cadlag paths, and are self-similar with index $H>0$.
\bigskip

\noindent (A-2): For each $t>0$ the distribution functions
$F(t,x)=P(X(t) \le x)$ have probability densities $f(t,x)$ that are symmetric about zero and such that
for each $t \in (0,\infty):$
$$ 
0<f (t,x)<\infty~ {\rm  {for~ all~}} x \in \mathbb{R}, ~{\rm{and}}~f(t,x)~{\rm{ is~ decreasing~ in~}} x \in [0,\infty).
$$

\begin{notat}
Let $\mathbb{Q}$ denote the rational numbers, $J=[1,2]$, $A=[1-\alpha^{*},\alpha^{*}], \frac{1}{2} <\alpha^{*} <1$, and for a subset $B$ of $\mathbb{R}$ we define $B_{\mathbb{Q}}= B \cap \mathbb{Q}$. 
\end{notat}

\begin{rem} The assumptions (A-1) and (A-2) are motivated by previous results in the area, see, for example, \cite{kkz} and \cite{kz-quant}, where some connections are made to the fact that we establish quantile CLTs which are uniform in $\alpha \in I$, where $I$ is an arbitrary closed interval in $(0,1)$. If $X$ is $H$-self-similar on $[0,\infty)$ with $H>0$, then it is easy to see we also have $X(0)=0$ with probability one.
\end{rem}

To obtain our results we also will have need for three less transparent assumptions. However, the first is studied in \cite{kkz}, with additional results in \cite{kz-quant}. The second is shown to mesh nicely with the empirical quantile CLTs in  \cite{kz-quant} via the use of self-similarity, and can be verified for many processes using the results of \cite{kkz} and \cite{kz-quant}. The third follows from basic properties of the input process $X$ in the eventual applications we obtain. However, all are nontrivial assumptions on the class of processes to which our results apply.
\bigskip

\noindent(A-3): The empirical CLT over $\mathcal{C}$ on $J$ with input data determined by i.i.d. copies of $\{X(t):t \ge 0\}$ holds in $\ell_{\infty}(J\times \mathbb{R}).$

\begin{rem}Since the process $\{X(t):t \ge 0\}$ is self-similar, it can be checked  
that the empirical quantile processes formed from i.i.d. copies of $X$ are scalable, with the same index,  and hence the interval $J= [1,2]$ can be replaced by any interval
$[a,b]$, where $0<a <b<\infty$. A similar fact also holds for the empirical CLT over $\mathcal{C}$ on $J$.
\end{rem}

\noindent(A-4): The empirical quantile CLT with input data determined by i.i.d. copies of $\{X(t):t \ge 0\}$ holds in $\ell_{\infty}(J\times I),$ where $I$ is an arbitrary closed interval of $(0,1)$.
The Gaussian limit process $\tilde G$ is centered and has covariance as given in Remark \ref{G-cov}  
below.

\begin{rem} The fact that the finite dimensional distributions of the empirical quantiles converge to centered Gaussian distributions with covariance structure as given in Remark \ref{G-cov} 
is a well-known classical result, see, for example, \cite{reiss-book}. Examples of the CLT in  $\ell_{\infty}(J \times I)$ as in (A-4) are much more recent, but many  can be found in \cite{kz-quant}.
\end{rem}

\begin{rem} If the densities $f(t,x)$ are such that
\begin{equation}\label{unif-cont-densities}
 \lim_{\delta\to 0}\sup_{t \in E} \sup_{|u-v|\le\delta}|f(t,u) -f(t,v)| =0, 
\end{equation}
and for every closed interval $I$ in $(0,1)$ there is an $\theta(I)>0$ such that 
\begin{equation}\label{inf-eq}
\inf_{t \in E,\alpha \in I,|x-\tau_{\alpha}(t)|\le \theta(I)} f(t,x) \equiv c_{I,\theta(I)}>0,
\end{equation}
then by Theorem 1 of [4] with $E=J$ the condition (A-3) implies that (A-4) holds. 
\end{rem}

The final assumption is that:
\bigskip

\noindent(A-5): There exists $0<\theta<\infty$ and $c_{\theta}<\infty$ such that 
\begin{align}\label{tail assumption}
P(\sup_{t\in J_{Q}}|X(t)|>u)\le c_{\theta}u^{-\theta} \text{ for all } u \in (0,\infty). 
\end{align}

\begin{thm}\label{scale-quant} Let $X=\{X(t): t \ge 0\}$ satisfy (A-$i$), $ i=1,2,3,4,5.$
Then, for each $T \in (0,\infty)$ the empirical quantile processes
\[\{\sqrt n (\tau_{\alpha}^n(t)- \tau_{\alpha}(t)) \colon\  n \geq 1\}\]
satisfy the CLT in $\ell_{\infty}([0,T] \times I)$ with centered Gaussian limit process 
\begin{equation}\label{quant-CLT-near-zero-2}
\left\{\tilde G(t,\alpha)\colon\  (t,\alpha) \in [0,T] \times I \right\},
\end{equation}
where $\tilde G(0,\alpha)=0, \alpha \in I$, 
\begin{equation}\label{quant-CLT-near-zero-3}
\tilde G(t,\alpha) = \frac{G(t,\tau_{\alpha}(t))}{f(t, \tau_{\alpha}(t))}, (t,\alpha) \in (0,T] \times I,
\end{equation}
and 
$\{G(t,x): t>0,x \in \mathbb{R}\}$ is the centered Gaussian process determined by the finite dimensional distributional limits of the empirical processes with covariance for $(s,x),(t,y) \in (0,\infty)\times\mathbb{R}$ 
\begin{equation}\label{G-covariance}
\E(G(t,x)G(t,y))= P(X_s \le x,X_t \le y)- P(X_s \le x)P(X_t \le y).
\end{equation}
\end{thm}

\begin{rem}\label{G-cov} Given (\ref{G-covariance}),  the covariance function of $\tilde G$ is such that 
$$
\E(\tilde G(0,\alpha) \tilde G(0,\beta))=0
$$ 
for $\alpha,\beta \in I$, and for $(s, \beta),(t,\alpha) \in (0,\infty) \times I$,
\begin{equation}\label{quant-CLT-near-zero-4}
\E(\tilde G(s,\beta) \tilde G (t,\alpha)) = \frac{P(X_s \le \tau_{\beta}(s),X_t \le \tau_{\alpha}(t))- \alpha \beta}{f(s,\tau_{\beta}(s))f(t, \tau_{\alpha}(t))}.
\end{equation}
\end{rem}

Our next result indicates that Theorem \ref{scale-quant} applies when the input process $X=\{X(t): t \ge 0\}$ is a sample continuous fractional Brownian motion with $X(0)=0$, or a symmetric stable process of index $ r \in (0,2)$ with stationary independent increments, cadlag paths, and $X(0)=0$. In either situation, the results of \cite{kz-quant} fail to apply directly, since the empirical CLT fails over $\mathcal{C}$ on $[0,T]$ for such processes, see \cite{kkz} and \cite{kz-quant}. In order to be precise about parameters for these processes we adopt the following notation.
 
 \begin{notat}
The process $\{X(t): t \ge 0\}$ is fractional Brownian motion with parameter $r \in (0,2)$ if it is a centered Gaussian process with covariance
\begin{equation}\label{fbm-covar}
\E(X(s)X(t))= \frac{1}{2}[t^r +s^r -|t-s|^r], s,t \ge 0.
\end{equation}
Furthermore, we assume the sample continuous version is chosen for use in our results involving fractional Brownian motions. The process $\{X(t): t \ge 0\}$ is a symmetric stable process of index $ r \in (0,2)$ with stationary independent increments, and $X(0)=0$ if there is a constant $c \in (0,\infty)$ such that for all $t \ge 0$ the characteristic function of $X(t)$ is given by
\begin{equation}\label{ssr-charact}
\E(\exp\{iuX(t\})= \exp \{-ct|u|^r\}, u \in \mathbb{R}.
\end{equation}
We assume that a cadlag path version is used in our results involving stable processes $X$ when $r \in (0,2)$.

\end{notat}

\begin{thm}\label{scale-quant-app} Let $X=\{X(t): t \ge 0\}$ be a symmetric $r$-stable process with stationary, independent increments, cadlag sample paths, or a centered sample continuous  fractional Brownian motion with parameter $r \in (0,2)$. Also, assume P(X(0)=0)=1, the  empirical quantile processes $\tau_{\alpha}^n(t)$ are built from i.i.d copies of $\{X(t): t \ge 0\}$ with sample paths those indicated for $X$, and $I$ is a closed subinterval of $(0,1)$. 
Then, for each $T \in (0,\infty)$ the empirical quantile processes
\[\{\sqrt n (\tau_{\alpha}^n(t)- \tau_{\alpha}(t)) \colon\  n \geq 1\}\]
satisfy the CLT in $\ell_{\infty}([0,T] \times I)$ with centered Gaussian limit process 
\begin{equation}\label{quant-CLT-app-1}
\left\{\tilde G(t,\alpha)\colon\  (t,\alpha) \in [0,T] \times I \right\},
\end{equation}
where $\tilde G(0,\alpha)=0, \alpha \in I$, and for $(s,\beta),(t,\alpha) \in (0,T]\times I$ its covariance is given by (\ref{quant-CLT-near-zero-4}). 
\end{thm}

\begin{rem}\label{densities} Under the conditions of Theorem \ref{scale-quant-app},
when $X$ is a fractional Brownian motion with parameter $r \in (0,2)$ and $t>0,\alpha \in (0,1)$, then 
$$
f(t,\tau_{\alpha}(t))= (2\pi)^{-\frac{1}{2}}t^{-\frac{r}{2}} \exp\{-\frac{\tau_{\alpha}^2(1)}{2} \},
$$
and for  $(s, \beta),(t,\alpha) \in (0,T] \times I$ the covariance $\E(\tilde G(s,\beta)\tilde G(t,\alpha))$ in (\ref{quant-CLT-near-zero-4}) is given by
\begin{equation}\label{quant-CLT-app-2}
2\pi(st)^{\frac{r}{2}} \exp\{\frac{\tau_{\alpha}^2(1) +\tau_{\beta}^2(1)}{2} \}[P(X_s \le \tau_{\beta}(s),X_t \le\tau_{\alpha}(t))- \alpha\beta].
\end{equation}
If $X$ is a symmetric $r$-stable process with stationary, independent increments with $X(0)=0$, then for $r \in (0,2), t>0, \alpha \in (0,1),$
$$
f(t,\tau_{\alpha}(t))= t^{-\frac{1}{r}}f(1,\tau_{\alpha}(1)),
$$
where $f(1,x), x \in \mathbb{R},$ is the strictly positive density of a symmetric r-stable random variable $X(1)$, i.e. for $c>0$ (\ref{ssr-charact}) implies

$$
2\pi f(1,x)= \int_{\mathbb{R}}\exp\{ixu-c|u|^r\}du= \int_{\mathbb{R}}\exp\{-c|u|^r\} {\rm{ cos}} (xu)du.
$$
Hence, for $(t,\alpha), (s,\beta) \in (0,T]\times I$, (\ref{quant-CLT-near-zero-4}) becomes
\begin{equation}\label{quant-CLT-app-3}
\E(\tilde G (s,\beta)\tilde G (t,\alpha)) = s^{\frac{1}{r}}t^{\frac{1}{r}}[\frac{P(X(s) \le \tau_{\beta}(s),X(t) \le \tau_{\alpha}(t))- \alpha \beta}{f(1,\tau_{\beta}(1))f(1, \tau_{\alpha}(1))}].
\end{equation}
Moreover, if $\alpha=\beta=\frac{1}{2}$, then the symmetry about zero of the distribution of $X(t)$ implies $\tau_{\frac{1}{2}}(t)=0$ for each $t>0$. Hence, for $s,t>0$, (\ref{quant-CLT-app-2}) implies
\begin{equation}\label{quant-CLT-app-4}
\E(\tilde G(s,\frac{1}{2})\tilde G(t,\frac{1}{2}))= 2\pi(st)^{\frac{r}{2}}[P(X_s \le 0,X_t \le 0)- \frac{1}{4}],
\end{equation}
and (\ref{quant-CLT-app-3}) implies
\begin{equation}\label{quant-CLT-app-5}
\E(\tilde G(s,\frac{1}{2})\tilde G(t,\frac{1}{2}))= (2\pi)^2(st)^{\frac{1}{r}}[\frac{P(X_s \le 0,X_t \le 0)- \frac{1}{4}}{(\int_{\mathbb{R}} \exp \{-c|u|^r \} du)^2}],
\end{equation}
Furthermore, since  $\tilde G (0,\alpha)=0, \alpha \in I$, we  also have (\ref{quant-CLT-app-2}),(\ref{quant-CLT-app-3}), (\ref{quant-CLT-app-4}), and (\ref{quant-CLT-app-5}) when  $(t,\alpha), (s,\beta)$ $\in [0,T]\times I$.
\end{rem}

Now we turn to the question of how special sample path properties of the input process $\{X_t: t \in E\}$ influence our quantile CLTs. To be more specific, recall that the CLT results we have established for empirical quantile processes, hold in the space $\ell_{\infty}(E\times I)$, and the limiting Gaussian process $\{\tilde G (t,\alpha): (t,\alpha) \in E\times I\}$, almost surely,  has a version with paths which are bounded and uniformly continuous with respect to its own $L_2$ distance $d_{\tilde G}$ on $E\times I$. In particular, this guarantees that the measure induced by the Gaussian process on $\ell_{\infty}(E\times I)$ is supported on the subspace $C_{L_2}(E\times I)$ of $\ell_{\infty}(E\times I)$, where the subscript $L_2$ is written to indicate the topology on $E \times I$ is that given by the Gaussian process $L_2$ distance. Hence, if $\alpha \in (0,1)$ is fixed,  $e_1$ is a metric on $E$, and the input process $\{X_t: t \in E\}$ is assumed sample continuous on $(E
 ,e_1)$, when does our quantile CLT hold on the space of $e_1$-continuous paths? If $E=[0,T]$ with metric the usual Euclidean distance,  and the input process has cadlag sample paths on $[0,T]$, a similar question can be asked if the quantile CLT holds in some related space of functions. Since processes with continuous paths or cadlag paths are typical of many examples throughout probability and statistics, these are natural questions,  but they also relate to the results of Swanson. That is, he established a CLT in the space of continuous functions on $[0,T]$ for the median process obtained from sample continuous Brownian motions in \cite{swanson-scaled-median}, and for other individual quantile levels $\alpha \in (0,1)$ in \cite{swanson-fluctuations}. These results will follow from our next theorem, and are established as applications following its proof.

Since the empirical quantile processes have jumps as $\alpha$ ranges over $(0,1)$, to state our theorem providing some facts related to these questions, we need the following function spaces. If $e_1$ is a metric on $E$ we set
\begin{equation}\label{quant-CLT-functions-1}
\mathbb{C}_{e_1}(E)=\{z: z {\rm {~is~continuous~on~}} (E,e_1)\},
\end{equation}
if $E=[0,T]$ we assume $e_1$ is the usual Euclidean distance and let
\begin{equation}\label{quant-CLT-functions-2}
\mathbb{D}_1([0,T])=\{z: z \rm {~is~cadlag~on~} [0,T]\},
\end{equation}
where right and left limits are taken with respect to $e_1$ on $[0,T]$, and for $I=[a,b]$ a closed subinterval of $(0,1)$ we set
\begin{equation}\label{quant-CLT-functions-3}
\mathbb{D}_2(I)=\{z: z {\rm {~left~continuous~on~}} (a,b], z {\rm {~ has~right~limits~on~}}[a,b)\},
\end{equation}
where right and left limits are taken with respect to the usual Euclidean distance $e_2$ on $I$. We also define the closed linear subspaces of $\ell_{\infty}(E\times I)$ given by
\begin{align}\label{quant-CLT-functions-4}
&\mathbb{C}_{e_1}(E) \otimes\mathbb{D}_2(I)=\{ f(\cdot,\alpha) \in \mathbb{C}_{e_1}(E) ~\forall \alpha \in I \\
&\quad {\rm{~and~}} f(t,\cdot) \in \mathbb{D}_2(I)~\forall t\in [0,T]\},\notag
\end{align}
and 
\begin{align}\label{quant-CLT-functions-5}
&\mathbb{D}_1([0,T])\otimes\mathbb{D}_2(I)=\{ f(\cdot,\alpha) \in \mathbb{D}_1([0,T]) ~\forall \alpha \in I\\
&\quad  \text{ and } f(t,\cdot) \in \mathbb{D}_2(I)~\forall t\in [0,T]\}.\notag
\end{align}

\begin{thm}\label{quant-CLT-paths}
Let $\{X_t: t \in E\}$ be the input process for the empirical quantile processes  defined for $t \in E, \alpha \in (0,1), n \ge 1,$ by
\begin{equation}\label{emp-quant-processes}
W_n(t,\alpha) := \sqrt n(\tau_{\alpha}^n(t) - \tau_{\alpha}(t)),
\end{equation}
and assume they satisfy the empirical quantile CLT in $\ell_{\infty}(E \times I)$ with Gaussian limit $\{\tilde G(t,\alpha): (t,\alpha) \in E\times I\}$, 
and that $\tau_{\alpha}(\cdot) \in \mathbb{C}_{e_1}(E)$ for every $ \alpha \in I$. Then, we have:

(i) If $\{X_t: t \in E\}$ has version with paths in $\mathbb{C}_{e_1}(E)$, then the empirical quantile CLT holds in the Banach subspace $\mathbb{C}_{e_1}(E) \otimes\mathbb{D}_2(I)$ of $\ell_{\infty}(E\times I)$. In particular, if $\alpha \in (0,1)$ is fixed, then the CLT will hold in the space of continuous functions $\mathbb{C}_{e_1}(E)$ with the topology that given by the sup-norm. 

(ii) If $E=[0,T]$ and $\{X_t: t \in E\}$ has version with paths in $\mathbb{D}_{1}([0,T])$, we have the empirical quantile CLT holding in the Banach subspace $\mathbb{D}_{1}([0,T]) \otimes\mathbb{D}_2(I)$ of  $\ell_{\infty}([0,T]\times I)$. Hence, if $\alpha \in (0,1)$ is fixed, then the CLT will hold in the space of functions $\mathbb{D}_{1}([0,T])$ with the topology that is given by the sup-norm. 
\end{thm}

\section{Proof of Theorem \ref{scale-quant}}

To prove this result we start with several lemmas using the earlier notation that $\mathbb{Q}$ denotes the rational numbers, $J=[1,2]$, $A=[1-\alpha^{*},\alpha^{*}], \frac{1}{2} <\alpha^{*} <1$, and for a subset $B$ of $\mathbb{R}$ we define $B_{\mathbb{Q}}= B \cap \mathbb{Q}$. 

\subsection{Some Lemmas and a Proposition}

The first lemma shows that if a process is scalable, then certain information on the interval $J$ yields information on the interval $[0,\delta]$. We phrase this in slightly more general terms, but ultimately it will be applied to quantile processes.

\begin{lem}\label{simplified homogeneous} Let $W$ be $H$-scalable. Assume that for some $q>0$ one has $\Ex \sup_{u\in J_{\mathbb{Q}},\alpha\in A_{\mathbb{Q}}}|W(u,\alpha)|^{q}<\infty$. Then, for every $0<\delta\in \mathbb{Q}$  
\begin{equation}
\E(\sup_{ u \in (0,\delta]_{\mathbb{Q}}, \alpha\in A_{\mathbb{Q}}}|W(u, \alpha)|^q])\le \dfrac{\delta^{Hq}}{1-2^{-Hq}}\E(\sup_{ u \in J_{\mathbb{Q}}, \alpha\in A_{\mathbb{Q}}}|W(u, \alpha)|^q]).
\end{equation}
\end{lem}

\begin{proof}
\begin{align*}
&\E\sup_{ s \in (0,\delta]_{\mathbb{Q}}, \alpha\in A_{\mathbb{Q}}}|W(s, \alpha)|^q]=\E \sup_{j\ge 1}\sup_{s\in (2^{-j}\delta, 2^{-(j-1)}\delta]_{\mathbb{Q}}, \alpha\in A_{\mathbb{Q}}}|W(s,\alpha)|^{q}\\
&\le\sum_{j=1}^{\infty}\ \E \sup_{s\in (2^{-j}\delta, 2^{-(j-1)}\delta]_{\mathbb{Q}}, \alpha\in A_{\mathbb{Q}}}|W(s,\alpha)|^{q}
=\sum_{j=1}^{\infty}\ \E \sup_{s\in (1,2]_{\mathbb{Q}}, \alpha\in A_{\mathbb{Q}}}|W(2^{-j}\delta s,\alpha)|^{q}\\
&=\sum_{j=1}^{\infty} (2^{-j}\delta)^{Hq}\ \E\sup_{s\in (1,2]_{\mathbb{Q}}, \alpha\in A_{\mathbb{Q}}}|W(s,\alpha)|^{q}=\frac{\delta^{Hq}}{1-2^{-Hq}}\E(\sup_{ u \in J_{\mathbb{Q}}, \alpha\in A_{\mathbb{Q}}}|W(u, \alpha)|^q]). 
\end{align*}
\end{proof}

Next we show that under the self-similarity assumption of Theorem \ref{scale-quant} we can apply Lemma \ref{simplified homogeneous} to  the sequence of empirical quantile processes 
\begin{equation}\label{emp-quant-processes-1}
W_{n}(t,\alpha)=\sqrt{n}\bigl(F^{-1}_{n,t}(\alpha)-F^{-1}_{t}(\alpha)\bigr),
\end{equation}
where $t\ge 0, \alpha\in(0,1)$. 
The bounds we obtain depend only on the scalability constant, $H$, and are uniform in $n$. 
For this purpose the next  lemma is also useful. 

\begin{lem}\label{quantile comparison} Let $X$ be an arbitrary random variable. If $q_{\alpha}(X)$ denotes any $\alpha$-quantile for $X$, then $-q_{1-\alpha}(-X)$ is also an $\alpha$-quantile for $X$. 
\end{lem}

\begin{proof} 
\begin{align*}
&P(X\ge -q_{1-\alpha}(-X))=P(-X\le q_{1-\alpha}(-X))\ge 1-\alpha\\
\intertext{and}
&P(X\le -q_{1-\alpha}(-X))=P(-X\ge q_{1-\alpha}(-X))\le \alpha. \qquad \qed
\end{align*}
\renewcommand{\qed}{}\end{proof}

\begin{prop}\label{tail bounds}
Let $\{X(t)\colon t\ge 0\}$ be a symmetric $H$-self-similar process such that (A-2), (A-3), and (A-5) hold, and assume the empirical quantile processes
are built from i.i.d. copies of $\{X(t): t \ge 0\}$. Then, there exists $n_{0}\ge 1$ such that 
for every  $\epsilon>0$ there exists $\delta>0$ with
\begin{equation}\label{emp-quant-processes-2}
\sup_{n \geq n_0} P(\sup_{t \in [0, \delta]_ Q, \alpha \in A_{\mathbb{Q}}} \sqrt n |F_{n,t}^{-1}(\alpha)-F_{t}^{-1}(\alpha)|> \epsilon) \leq \epsilon.
\end{equation}
If we also assume (A-1), 
then there exists an integer $n_0 \ge 1$ such that for every $\epsilon>0$ there is  a $\delta>0$ satisfying
\begin{equation}\label{cont-zero-1}
\sup_{ n \ge n_0}P(\sup_{t \in [0,\delta], \alpha \in A} \sqrt n |F_{n,t}^{-1}(\alpha)-F_{t}^{-1}(\alpha)| >\epsilon)\le \epsilon. 
\end{equation} 
\end{prop}

\begin{proof} First we note that the $H$-self-similarity of each of the i.i.d. processes, $\{X_{j}(t)\},$ immediately implies scalability with the same $H$ for each of the processes 
\[\{W_n(t,\alpha)= \sqrt n(F_{n,t}^{-1}(\alpha)-F_{t}^{-1}(\alpha))\colon t\in [0,\infty), \alpha\in (0,1)\}.
\] 
We'll obtain bounds on 
\begin{equation}\label{emp-quant-processes-3}
P(\sup_{t\in J_{\mathbb{Q}},\alpha\in A_{\mathbb{Q}}} \sqrt n|F_{n,t}^{-1}(\alpha)-F_{t}^{-1}(\alpha)|>u),
\end{equation}
strong enough to yield an $n_{0}$ for which 
\begin{equation}
\sup_{n \ge n_0} \E[\sup_{t\in J_{\mathbb{Q}}, \alpha\in A_{\mathbb{Q}}} \sqrt n |F_{n,t}^{-1}(\alpha)-F_{t}^{-1}(\alpha)|] <\infty. \notag
\end{equation}
At this point for any $q \in (0,1]$ and $n \ge n_0$, we can apply Lemma \ref{simplified homogeneous} to obtain the bound
\begin{align}\label{emp-quant-processes-4}
\E(\sup_{ t \in (0,\delta]_{\mathbb{Q}}, \alpha\in A_{\mathbb{Q}}}&|W_n(t,\alpha)|^{q})\le \dfrac{\delta^{Hq}}{1-2^{-Hq}}
&\E(\sup_{ t \in J_{\mathbb{Q}}, \alpha\in A_{\mathbb{Q}}}|W_n(t,\alpha)|^{q})<\infty.
\end{align}
An application of Chebyschev's inequality and $\delta=\delta(\epsilon)$ sufficiently small will then yield (\ref{emp-quant-processes-2}), which proves the proposition. 
\par

We break the proof into two parts. The first part covers the case when we have a lower bound on the densities of $X_{t}$ for $t\in J=[1,2]$. In the second part we take care of the remaining case. For notational convenience let $\eta=F_{t}^{-1}(\alpha), \zeta=\dfrac{u}{\sqrt{n}}+\eta$.
Now for $u\ge 0$ we have
  \begin{align}&P(\sup_{t\in J_{\mathbb{Q}}, \alpha\in A_{\mathbb{Q}}}W_n(t,\alpha)>u)
\!=\!P(\exists (t,\alpha) \!\!\in J_{\mathbb{Q}}\times A_{\mathbb{Q}}, \!\sum_{j=1}^{n}I_{X_{j}(t)>\zeta}\ge n(1-\alpha))\!=\!\notag\\
&P(\exists (t,\alpha) \!\in \!J_{\mathbb{Q}}\!\times \!A_{\mathbb{Q}},  \!\sum_{j=1}^{n}\bigl(I_{X_{j}(t)\!>\!\zeta}\!-\!P(X(t)\! >\!\zeta)\bigr)\!
\ge\! n\bigl[(\!1\!-\!\alpha)\!\! -\!\! P(X(t)>\zeta)\bigr])\notag\\
&=P(\exists (t,\alpha) \in J_{\mathbb{Q}}\times A_{\mathbb{Q}}, \sum_{j=1}^{n}\bigl(I_{X_{j}(t) >\zeta}\!-\!P(X(t)\! >\!\zeta)\bigr)\!\ge\! nP(\eta\!<\!X(t)\le \zeta)).\label{part 1}
\end{align}

Since  
(A-2) holds, for each $t>0$ the density $f(t,x)$ of $X(t)$ is symmetric about $0$, and  decreasing on $[0,\infty)$. Therefore, $\{X(t): t \ge 0\}$ being $H$-self-similar with $H \in (0,\infty)$, and $1\le  t \le 2$, imply
\begin{align*}{P}&(\eta<X(t)\le \zeta)=P(\eta t^{-H}\le X(1)\le \zeta t^{-H})\\
&=\int_{\eta t^{-H}}^{\zeta t^{-H}}f(1,x)dx
\ge \bigl(\inf_{\eta t^{-H}\le x\le \zeta t^{-H}}f(1,x)\bigr)(\zeta-\eta) t^{-H}\\
&\ge f(1, t^{-H} F_{t}^{-1}(\alpha^{*})+\dfrac{u}{\sqrt{n}})\ \dfrac{u}{\sqrt{n}} t^{-H}\ge f(1,F_{1}^{-1}(\alpha^{*})+\dfrac{u}{\sqrt{n}})\ 2^{-H}\dfrac{u}{\sqrt{n}}. 
\end{align*}
So, if  $0 \le \dfrac{u}{\sqrt{n}}\le C$, 
\begin{align*}
&\{\sqrt{n}\sup_{t\in J_{\mathbb{Q}}, \alpha\in A_{\mathbb{Q}}}\bigl(F_{n,t}^{-1}(\alpha)-F_{t}^{-1}(\alpha)\bigr)>u\}\\
&\subseteq \{\dfrac1{{n}}\|\sum_{j=1}^{n}\bigl(I_{X_{j}(t)>y}-P(X(t)>y)\bigr)\|_{J_{\mathbb{Q}}\times \mathbb{R}}
\ge  f(1,F_{1}^{-1}(\alpha^{*})+C)\ 2^{-H}\dfrac{u}{\sqrt{n}}\}.
\end{align*}
Now, for $t \in J_{\mathbb{Q}}$ fixed, the continuity of $P(X(t) >y)$ in $y$ and the right continuity of $I_{X_j(t) >y}$ in $y$ for $1 \le j \le n$  implies
\begin{align*}
&\frac{\|\sum_{j=1}^{n}\bigl(I_{X_{j}(t)>y}-P(X(t)>y)\bigr)\|_{J_{\mathbb{Q}}\times \mathbb{R}}}{\sqrt n}
=\frac{\|\sum_{j=1}^{n}\bigl(I_{X_{j}(t)>y}-P(X(t)>y)\bigr)\|_{J_{\mathbb{Q}}\times \mathbb{Q}}}{\sqrt n}
\end{align*}
with probability one.
Hence, if $D:= 2^{-H}f(1,F_{1}^{-1}(\alpha^{*})+C)$, we have for $0 \le u\le C\ \sqrt{n}$, 
\begin{align}\label{emp upper bound}P&(\sqrt{n}\sup_{t\in J_{\mathbb{Q}}, \alpha\in A_{\mathbb{Q}}}\bigl(F_{n,t}^{-1}(\alpha)-F_{t}^{-1}(\alpha)\bigr)>u)\notag\\
&\le P(\dfrac1{\sqrt{n}}\|\sum_{j=1}^{n}\bigl(I_{X_{j}(t)>y}-P(X(t)>y)\bigr)\|_{J_{\mathbb{Q}}\times \mathbb{Q}}\ge  Du)
\end{align}
Since the summands, $I_{X_{j}(t)>y}-P(X(t)>y)$ are bounded by $1$, and the CLT  over $\mathcal{C}$ on $J$ given by assumption (A-3) implies stochastic boundedness of the normalized norm in (\ref{emp upper bound}), we can use a result of  Hoffman-J\o rgensen, see pp. 164-5 of  \cite{hoff}, 
to obtain for any $q>0$, 
\[B_{q}:=\sup_{n}\E \dfrac1{\sqrt{n}}\|\sum_{j=1}^{n}\bigl(I_{X_{j}(t)>y}-P(X(t)>y)\bigr)\|_{J_{\mathbb{Q}}\times \mathbb{Q}}^{q}<\infty.
\]
Therefore, for $0 \le  u \le C\, \sqrt n,$
\begin{align}\label{small values}P&(\sqrt{n}\sup_{t\in J_{\mathbb{Q}}, \alpha\in A_{\mathbb{Q}}}\bigl(F_{n,t}^{-1}(\alpha)-F_{t}^{-1}(\alpha)\bigr)>u)\le B_{q}\dfrac{1}{(Du)^{q}}.
\end{align}

Now we deal with the case $u\ge C\sqrt{n}$. 
In the  computation below we don't use the particular form of the quantiles, $F_{n,t}^{-1}(\alpha), F_{t}^{-1}(\alpha)$, only the fact that they are quantiles. Note (\ref{other quantile}) below. 
Hence, by Lemma \ref{quantile comparison} 
\begin{align}
&P(\sqrt{n}\sup_{t\in J_{\mathbb{Q}},\alpha\in A}|F_{n,t}^{-1}(\alpha)-F_{t}^{-1}(\alpha)|>u)\notag\\
\le~ &P(\sqrt{n}\sup_{t\in J_{\mathbb{Q}},\alpha\in A_{\mathbb{Q}}}\bigl(F_{n,t}^{-1}(\alpha)-F_{t}^{-1}(\alpha)\bigr)>u)\notag\\
+~ &P(\sqrt{n}\sup_{t\in J_{\mathbb{Q}},\alpha\in A_{\mathbb{Q}}}\bigl(F_{t}^{-1}(\alpha)-F_{n,t}^{-1}(\alpha)\bigr)>u)\notag\\
 =~ &P(\sqrt{n}\sup_{t\in J_{\mathbb{Q}},\alpha\in A_{\mathbb{Q}}}\bigl(F_{n,t}^{-1}(\alpha)-F_{t}^{-1}(\alpha)\bigr)>u)\notag\\
 +~ &P(\sqrt{n}\sup_{t\in J_{\mathbb{Q}},\alpha\in A_{\mathbb{Q}}}\bigl(-q_{t}(1-\alpha)+q_{n,t}(1-\alpha)\bigr)>u).\label{other quantile}
 \end{align} 
Thus, the second term can be treated the same as the first term. 
For the first term we have 
\begin{align}\label{minmax}P&(\sqrt{n}\sup_{t\in J_{\mathbb{Q}},\alpha\in A_{\mathbb{Q}}}\bigl(F_{n,t}^{-1}(\alpha)-F_{t}^{-1}(\alpha)\bigr)>u)\notag\\
&=P(\exists t\in J_{\mathbb{Q}},\alpha\in A_{\mathbb{Q}}, F_{n,t}^{-1}(\alpha)>\zeta)\notag\\
&\le P(\exists t\in J_{\mathbb{Q}}, \alpha\in A_{\mathbb{Q}}, \exists I, \#I= \lfloor n(1-\alpha)\rfloor, X_{j}(t)\notag\\
&\quad >\zeta, \forall j\in I)\notag\\
&\le P(\exists t\in J_{\mathbb{Q}}, \alpha\in A_{\mathbb{Q}}, \exists I, \#I= \lfloor n(1-\alpha^{*}) \rfloor, X_{j}(t)\notag\\
&\quad >\zeta, \forall j\in I),\notag
\end{align}
and again by Lemma \ref{quantile comparison}, since  $F_{t}^{-1}(\alpha)\ge F_{t}^{-1}(1-\alpha^{*})~ \text{for all}~ \alpha \in A_{\mathbb{Q}}$,
\begin{align}
&\le \binom{n}{\lfloor n(1-\alpha^{*})\rfloor}P(\exists t\in J_{\mathbb{Q}}, X_{j}(t)>\dfrac{u}{\sqrt{n}}-F_{t}^{-1}(\alpha^{*}),\notag\\
&\quad  j=1,\ldots, \lfloor n(1-\alpha^{*})\rfloor) \notag\\
&\le\binom{n}{\lfloor n(1-\alpha^{*})\rfloor}P(\min_{j\le \lfloor n(1-\alpha^{*})\rfloor }\|X_{j}\|_{J_{\mathbb{Q}}}>\dfrac{u}{\sqrt{n}}-F_{1}^{-1}(\alpha^{*}))\notag\\
&\le\binom{n}{\lfloor n(1-\alpha^{*})\rfloor}\bigl[P(\|X\|_{J_{\mathbb{Q}}}>\dfrac{u}{\sqrt{n}}-F_{1}^{-1}(\alpha^{*}))\bigr]^{\lfloor n(1-\alpha^{*})\rfloor }\notag\\
&\le \bigl[\dfrac{e}{1-\alpha^{*}}P(\|X\|_{J_{\mathbb{Q}}}>\dfrac{u}{\sqrt{n}}-F_{1}^{-1}(\alpha^{*}))\bigr]^{\lfloor n(1-\alpha^{*})\rfloor}.
\end{align}
The last inequality in (\ref{minmax}) follows  since $r_n=\lfloor (1-\alpha^{*})n\rfloor$ and $ n\ge \frac{1}{1-\alpha^{*}}$ implies
$$
 \binom{n}{r_n} \le \big(\frac{e}{1-  \alpha^{*}}\big)^{r_n}.
$$
This can be seen
by observing that 
$$
\ln r! \ge \int_1^r \ln xdx= r \ln r -r +1, 
$$
and  $e^{\frac{1}{r}} r \ge r+1$ for all $r \ge 1,$ which together imply
$$
 r! \ge \big(\frac{e^{\frac{1}{r}} r}{e}\big)^r \ge \big(\frac{r+1}{e}\big)^r.
$$ 
Therefore, since $r_n \ge 1$ we have 
$$
\binom{n}{r_n} \le \frac{n^{r_n}}{r_n!}  \le \big(\frac{ne}{r_n+1} \big)^{r_n} \le \big(\frac{e}{1- \alpha^{*}} \big)^{r_n}.
$$

Now, by our tail assumption (A-5) we have by (\ref{minmax}), if $\dfrac{u}{\sqrt{n}}\ge C\ge 2F_{1}^{-1}(\alpha^{*})$ (see the definition of $C$ below), we have for $\lambda_{\theta}^{\theta}:=\dfrac{2^{\theta}e c_{\theta}}{1-\alpha^{*}}$, 
\begin{align}\label{final minmax}
&P(\sqrt{n}\sup_{t\in J_{\mathbb{Q}},\alpha\in A_{\mathbb{Q}}}\bigl(F_{n,t}^{-1}(\alpha)-F_{t}^{-1}(\alpha)\bigr)>u)\notag\\
&\le \bigl[\dfrac{e}{1-\alpha^{*}}P(\|X\|_{J_{\mathbb{Q}}}>\dfrac{u}{2\sqrt{n}})\bigr]^{\lfloor n(1-\alpha^{*})\rfloor}\notag\\
&\le\bigl[\dfrac{e}{1-\alpha^{*}}c_{\theta}(\dfrac{2\sqrt{n}}{u})^{\theta}\bigr]^{\lfloor n(1-\alpha^{*})\rfloor }=\bigl[\frac{\lambda_{\theta} \sqrt n}{u}\bigr]^{\theta \lfloor n(1-\alpha^{*})\rfloor }.
\end{align}

Therefore, taking
$u/\sqrt{n} \ge C \equiv 2\lambda_{\theta} \vee 2F_1^{-1}(\alpha^{*})$,  (\ref{final minmax}) implies
\begin{align}\label{final minmax-1}P&(\sqrt{n}\sup_{t\in J_{\mathbb{Q}},\alpha\in A_{\mathbb{Q}}}\bigl(F_{n,t}^{-1}(\alpha)-F_{t}^{-1}(\alpha)\bigr)>u)\le\bigl[\frac{\lambda_{\theta} \sqrt n}{u}\bigr]^{\theta \lfloor n(1-\alpha^{*})\rfloor }\notag\\
 &\phantom{*************} \le 
2^{-(\theta \lfloor n(1-\alpha^{*})\rfloor -2)}(\frac{\lambda_{\theta} \sqrt n}{u})^2,
\end{align}
and hence $n$ sufficiently large, say $n \ge n_0$, implies
\begin{align}\label{final minmax-2}P(\sqrt{n}\sup_{t\in J_{\mathbb{Q}},\alpha\in A_{\mathbb{Q}}}\bigl(F_{n,t}^{-1}(\alpha)-F_{t}^{-1}(\alpha)\bigr)>u)\le u^{-2}.
\end{align}

Since the same estimates apply to the second term in (\ref{other quantile}), we have by putting the two parts together that
\begin{align}
\E&[\sup_{t\in J_{\mathbb{Q}}, \alpha\in A_{\mathbb{Q}}} \sqrt n |F_{n,t}^{-1}(\alpha)-F_{t}^{-1}(\alpha)|]\notag\\
&\le 2 \int_{0}^{\infty}P(\sqrt{n}\sup_{t\in J_{\mathbb{Q}}, \alpha\in A_{\mathbb{Q}}}\bigl(F_{n,t}^{-1}(\alpha)-F_{t}^{-1}(\alpha)\bigr)>u)\, du\notag\\
&\le 2[1+\int_{1}^{C\sqrt{n}}P(\sqrt{n}\sup_{t\in J_{\mathbb{Q}}, \alpha\in A_{\mathbb{Q}}}\bigl(F_{n,t}^{-1}(\alpha)-F_{t}^{-1}(\alpha)\bigr)>u)\, du\notag\\
&\phantom{...........}+\int_{C\sqrt{n}}^{\infty}P(\sqrt{n}\sup_{t\in J_{\mathbb{Q}}, \alpha\in A_{\mathbb{Q}}}\bigl(F_{n,t}^{-1}(\alpha)-F_{t}^{-1}(\alpha)\bigr)>u)\, du]\notag\\
&\le 2[1+\dfrac{B_{2}}{D^{2}}\int_{1}^{C\sqrt{n}}\dfrac1{u^{2}}\, du +\int_{C\sqrt{n}}^{\infty}\dfrac1{u^{2}}\, du]  < \infty,
\end{align}
provided $n\ge n_0$ is sufficiently large, $B_2$ and $D$ are as in (\ref{small values}), and $C \equiv 2\lambda_{\theta} \vee 2F_1^{-1}(\alpha^{*})$. Thus the hypotheses in Lemma \ref{simplified homogeneous} are uniformly satisfied for $n \ge n_0$ and $q=1$. Hence, as indicated following  (\ref{emp-quant-processes-4}), (\ref{emp-quant-processes-2}) is proved and the first claim of the proposition is proved.

The second claim of the of the proposition in (\ref{cont-zero-1}) follows immediately from the next lemma. Hence the proposition is proved once this lemma  is established.
\end{proof}

The next lemma is important in that it allows us to switch back and forth between supremums over countable and uncountable parameter sets as in its application to Proposition 1. The sets $A$ and $A_{\mathbb{Q}}$ are as above, but $[0,T]_{\mathbb{Q}}$ also includes the point $T$, even if it is irrational. An extension of Lemma 3 follows by combining Theorem A and Lemma A in the comments at end of the paper. This refinement is due to the referee.

\begin{lem}\label{ctble-unctble-sups}
Let $X=\{X(t)\colon t\ge 0\}$ be a $H$-self-similar process with cadlag paths, $0<T< \infty$, and assume
the density, $f(1, \cdot),$ of $X(1)$ is strictly positive. Then, the  empirical quantile process $\tau_{\alpha}^n(t)$ built from i.i.d. copies of $\{X(t): t \ge 0\}$ with cadlag paths on a complete probability space has right continuous paths on $[0,T)$ with probability one, and is such that
\begin{equation}\label{ctble-unctble-0}
P(\sup_{t \in [0,T], \alpha \in A} |\tau_{\alpha}^n(t) - \tau_{\alpha}(t)| = \sup_{t \in [0,T]_{\mathbb{Q}}, \alpha \in A_{\mathbb{Q}}} |\tau_{\alpha}^n(t)) - \tau_{\alpha}(t)|)=1.
\end{equation} 
With probability one, $\tau_{\alpha}^n(t)$ also has left hand limits in $t \in (0,T]$, and if the input process $X$ is sample path continuous then $\tau_{\alpha}^n(t)$ is also path continuous in $t \in [0,T]$ with probability one.
Moreover, for each $t \in [0,T]$ and $n \geq 1,$ with probability one the empirical quantile process  $\tau_{\alpha}^n(t)$ is left continuous and has right limits in $\alpha \in (0,1)$.
\end{lem}

\begin{proof} 
Since $\{X(t): t \ge 0\}$ is $H$-self-similar with $H>0$, we have $P(X(0)=0)=1,$ and for $t>0$ 
\begin{equation}\label{dist-scale}
F(t,x)=F(1,t^{-H}x),
\end{equation}
where $F(t,\cdot)$ is the distribution of $X(t)$. In addition, if $t>0$, then $X(t)$ has a strictly positive density since
\begin{equation}\label{density-scale}
f(t,x)=t^{-H}f(1,t^{-H}x),
\end{equation}
and
$F(t,x)$ is strictly increasing and  continuous in $ x \in \mathbb{R}$. Thus $\tau_{\alpha}(t) = F_t^{-1}(\alpha)$ is continuous in $\alpha \in (0,1)$, and by its definition, $ \tau_{\alpha}^n(t)= F_{n,t}^{-1}(\alpha)$ is left continuous with right limits in $\alpha \in (0,1)$. Moreover, for every $\alpha \in (0,1)$, $P(\tau_{\alpha}(0) = \tau_{\alpha}^n(0)=0)=1$ since  $P(X_j(0)=0)=1, j\ge1.$ Therefore, we have 
\begin{equation}\label{ctble-unctble-2}
P(\sup_{t \in [0,T], \alpha \in A} |\tau_{\alpha}^n(t) - \tau_{\alpha}(t)| = \sup_{t \in [0,T], \alpha \in A_{\mathbb{Q}}} |\tau_{\alpha}^n(t)) - \tau_{\alpha}(t)|)=1.
\end{equation} 

We also have $\tau_{\alpha}(\cdot)$ continuous in $t$ on $[0,\infty)$, since scaling easily implies $\tau_{\alpha}(t) = t^{H}\tau_{\alpha}(1)$ for all $t \ge 0$. Thus (\ref{ctble-unctble-0})
follows from (\ref{ctble-unctble-2}) provided we show that $\tau_{\alpha}^n(t)$ is right continuous on $[0,T)$ with probability one, i.e. we then would have
\begin{equation}\label{ctble-unctble-3}
P(\sup_{t \in [0,T], \alpha \in A_{\mathbb{Q}}} |\tau_{\alpha}^n(t) - \tau_{\alpha}(t)| = \sup_{t \in [0,T]_{\mathbb{Q}}, \alpha \in A_{\mathbb{Q}}} |\tau_{\alpha}^n(t)) - \tau_{\alpha}(t)|)=1.
\end{equation}

To verify the right continuity of $\tau_{\alpha}^n(\cdot)$, and (\ref{ctble-unctble-3}), we use the right continuity of the paths of the processes $X_1,\cdots,X_n$. We do this through the following observation.
That is, given real numbers $\{x_1,\cdots,x_n\}$, let $x_{(1)},\cdots,x_{(n)}$ be an ordering of these numbers such that $x_{(1)}\le \cdots \le x_{(n)}$. In case the numbers $\{x_1,\cdots,x_n\}$ are distinct, this ordering is unique, and when there are ties, we choose the ordering based on the priority of the original index among the tied numbers. Then, we refer to $x_{(1)},\cdots,x_{(n)}$ as the order statistics of $\{x_1,\cdots,x_n\}$, and for $k=1,\cdots,n$ we have
$$
x_{(k)}= \min_{J\in \mathcal{J}_k} \max_{i \in J}x_i,
$$
where $ \mathcal{J}_k$ denotes all subsets of $\{1,\cdots,n\}$ with $k$ or more elements. Furthermore, given $\{x_1,\cdots,x_n\}$ and $\{y_1,\cdots,y_n\}$, it follows that
\begin{equation}\label{ctble-unctble-4}
\max_{1 \le k \le n}|x_{(k)} -y_{(k)}| \le \max_{1 \le k \le n}|x_k-y_k|, 
\end{equation}
since for all $\delta\ge  \max_{1 \le k \le n}|x_k-y_k|$
$$
y_{(k)} - \delta = \min_{J\in \mathcal{J}_k}\max_{i\in J}y_{i}-\delta \le x_{(k)} = \min_{J \in \mathcal{J}_k}\max_{i\in J}x_{i}\le \min_{J\in \mathcal{J}_k}\max_{i\in J}y_{i}+\delta=y_{(k)}+\delta,
$$
which implies (\ref{ctble-unctble-4}) holds.

Since the i.i.d. processes $X_1,\cdots,X_n$ are cadlag on $[0,\infty)$ with probability one, there is a set $\Omega_1 \subseteq \Omega$ such that $P(\Omega_1)=1$ and for every $t \in [0,T), \epsilon>0,$ there is a $\delta=\delta(\omega,t,\epsilon,n)>0$ such that $\omega \in \Omega_1$ implies 
$$
\sup_{1 \leq j \leq n, t\le s \le (t+\delta) \wedge T} |X_j(s) - X_j(t)| \leq \epsilon.
$$
Therefore, (\ref{ctble-unctble-4}) implies the order statistics $X_{(1)}(s) \leq \cdots \leq X_{(n)}(s)$ and $X_{(1)}(t) \leq \cdots \leq X_{(n)}(t)$ obtained from $\{X_1(s),\cdots,X_n(s)\}$ and  $\{X_1(t),\cdots$, $X_n(t)\}$ are such that
\begin{equation}\label{ctble-unctble-5}
\sup_{1 \leq j \leq n, t\le s \le (t+\delta) \wedge T} |X_{(j)}(s) - X_{(j)}(t)| \leq \epsilon. 
\end{equation}
Since $\epsilon>0$ is arbitrary, we thus have that the order statistic processes $\{X_{(j)}(t): t \in [0,T)\}, j=1,\cdots,n,$ are right continuous on $\Omega_1$, and hence with probability one.

Now for $0<\alpha <1, n \geq 1, t \in [0,\infty)$ we have $\tau_{\alpha}^n(t) = \inf\{x: F_n(t,x) \geq \alpha\}$, and hence 
$$
\tau_{\alpha}^n(t) = X_{(j(\alpha))}(t), 
$$
where $j(\alpha)= \min\{ k: 1 \leq k \leq n, k/n  \geq \alpha\}$ is independent of $t \in E$. 
Thus for all $\omega \in \Omega_1$ 
we have $\tau_{\alpha}^n(t)$ right continuous in $t \in [0,T)$. 

To verify $\tau_{\alpha}^n(t)$ also has left hand limits on $(0,T]$ with probability one, and that its paths are cadlag, observe that for $s,t \in (0,T]$ the previous argument implies
$$
|\tau_{\alpha}^n(s) - \tau_{\alpha}^n(t)|= |X_{(j(\alpha))}(s)- X_{(j(\alpha))}(t)| \le\sup_{1 \le j \le n}|X_j(s)-X_j(t)|.
$$
Since the paths of each $X_j$ are cadlag  on $\Omega_1$, with $P(\Omega_1)=1$, for all $r \in (0,T]$ 
$$
\lim_{s \uparrow r} \sup_{s<t<r}|\tau_{\alpha}^n(s) - \tau_{\alpha}^n(t)|\le  \lim_{s\uparrow r} \sup_{s<t<r}\sup_{1 \le j \le n}|X_j(s)-X_j(t)|=0,
$$
which implies $\tau_{\alpha}^n(\cdot)$ has left hand limits on $\Omega_1$ at every point $r \in (0,T]$. The continuity properties claimed also now follow, and hence the lemma is proven.
\end{proof}

\subsection{Proof of Theorem \ref{scale-quant}} 

\begin{proof}
Since we are assuming (A-$i$) for $i=1,2,3,4,5,$ the conclusions of Lemmas 1,2,3, and Proposition 1,  all hold for the proof of Theorem \ref{scale-quant}. Hence, let
$$
W_n(t,\alpha)= \sqrt n((\tau_{\alpha}^n(t)- \tau_{\alpha}(t)), t \in [0,T], \alpha \in I,n \geq 1.
$$
Then, $P(W_n(0,\alpha)=0)=$1 for $\alpha \in I, n \ge 1,$ and for $t\in (0,T], \alpha \in I,$ well-known finite dimensional results imply the finite dimensional distributions of $W_n$ converge to the centered Gaussian distributions given by the covariance function in Remark \ref{G-cov} and (\ref{quant-CLT-near-zero-4}). Of course, this also holds by scaling and (A-4), but the covariance structure is determined by the convergence of the finite dimensional distributions, which is far less demanding than the assumption (A-4). Hence Theorem 1.5.4 and Theorem 1.5.6 of \cite{vw} combine to imply the quantile processes $\{W_n: n\ge 1\}$ satisfy the CLT in $\ell_{\infty}([0,T]\times I)$, where the limiting centered Gaussian process has  the covariance indicated,
provided for every $\epsilon >0, \eta>0 $ there is a partition
\begin{equation}\label{quant-CLT-near-zero-5}
[0,T] \times I= \cup_{i=1}^k E_i
\end{equation}
such that
\begin{equation}\label{quant-CLT-near-zero-6}
\limsup_{n \rightarrow \infty}P^{*}( \sup _{1 \le i \le k} \sup_{(t,\alpha),(s,\beta) \in E_i}|W_n(t,\alpha) - W_n(s,\beta)| > \epsilon) \le \eta,
\end{equation}
where $P^*(A)$ denotes the outer $P$-probability of an event $A$.
Since $I$ is a closed subinterval of $(0,1)$, there is an $\alpha^{*} \in (\frac{1}{2}, 1)$ such that $I \subseteq A= [1-\alpha^{*}, \alpha^{*}]$. For $\delta>0$ and $E_1=[0,\delta] \times I$ observe that
\begin{align*}
&P^{*}( \sup_{(t,\alpha),(s,\beta) \in E_1}|W_n(t,\alpha) - W_n(s,\beta)| >\frac{ \eta \wedge \epsilon}{2})\\
 \le~ &2P^{*}(\sup_{ s \in [0,\delta], \alpha \in I} |W_n(s,\alpha)| >  \frac{ \eta \wedge \epsilon}{4}).
\end{align*}
Hence, (\ref{cont-zero-1}) implies there is a $\delta=\delta(\frac{\eta \wedge \epsilon}{4})$ such that 
\begin{equation}\label{quant-CLT-near-zero-7}
\limsup_{n \rightarrow \infty} P^{*}( \sup_{(s,\alpha),(t,\beta) \in E_1}|W_n(s,\alpha) - W_n(t,\beta)| >\frac{ \eta \wedge \epsilon}{2}) \le 2(\frac{ \eta \wedge \epsilon}{4})\le \frac{\eta}{2}.
\end{equation}

Now (A-4) and scaling implies the CLT for $\{W_n(t,\alpha): (t, \alpha) \in [\delta,T]\times I\} $ in $\ell_{\infty}([\delta,T] \times I)$, and hence Theorem 1.5.4 of \cite{vw} implies that there is a partition  $[\delta, T]\times I= \cup_{i=2}^k E_i$ such that 
\begin{equation}\label{quant-CLT-near-zero-8}
\limsup_{n \rightarrow \infty}P^{*}( \sup _{2 \le i \le k} \sup_{(s,\alpha),(t,\beta) \in E_i}|W_n(s,\alpha) - W_n(t,\beta)| >\frac{ \eta \wedge \epsilon }{2}) \le \frac{ \eta \wedge \epsilon}{2}.
\end{equation}
Combining (\ref{quant-CLT-near-zero-7}) and (\ref{quant-CLT-near-zero-8}) we have (\ref{quant-CLT-near-zero-5}) and (\ref{quant-CLT-near-zero-6}), and hence the theorem is proved.
\end{proof}

\section{Proof of Theorem \ref{scale-quant-app}}
\begin{proof}
If $X=\{X(t):t \ge 0\}$, then $X$ is $H$-self-similar for $H=\frac{1}{r}$ for the stable processes indicated, and for $H=\frac{r}{2}$ when $X$ is a fractional Brownian motion. Given our assumptions on the sample paths we thus have (A-1) holding. The strict positivity of the densities for $t>0$ in assumption (A-2) is immediate by using the self-similarity property and the comments in Remark \ref{densities}. Also, in the stable cases, since the density, $f(1,\cdot)$, of $X(1)$, is symmetric about $0$ and unimodal (\cite{yamazato-unimodal}), it is decreasing away from the origin. Of course, this last claim is trivial for the fractional Brownian motions, so (A-2) holds.

The more subtle assumptions of (A-3) and (A-4) follow by applying results in \cite{kz-quant} for the processes indicated. In particular, the empirical CLT in (A-3) holds when the input process is a fractional Brownian motion by applying Proposition 2 of \cite{kz-quant}, and for the stable processes, because of the independent increments, from Theorem 2 of \cite{kz-quant}. The self-similarity assumption and that (A-2) holds allows one to check the assumptions on the densities $f(t,\cdot), t \in J,$ of Theorem 1 in \cite{kz-quant} when $E=J$, and since we have shown (A-3) holds, Theorem 1 of \cite{kz-quant} implies (A-4) for the processes indicated.

The assumption (A-5) for stable processes with cadlag sample paths follows from Proposition 5.6 of
\cite{led-tal-book}, and for sample continuous Gaussian processes on a compact metric space it is an immediate consequence of the Fernique-Landau-Shepp result, see, for example Proposition A.2.3 of  \cite{vw}.
Therefore, assumptions (A-$i$) hold for $i=1,2,3,4,5,$ and hence Theorem \ref{scale-quant-app} follows from Theorem \ref{scale-quant}.
\end{proof}

\section{Proof of Theorem \ref{quant-CLT-paths} and Some Applications}\label{sec:5}

We start with a lemma.

\begin{lem}\label{weak-conv-paths}
Let $W_n, n \ge 1,$ take values only in the closed subset $D$ of $\ell_{\infty}(E \times I)$, $G$ be a stochastic processes on $(\Omega,\mathcal{F}, P)$ with sample paths in $\ell_{\infty}(E\times I),$ and assume $W_n(\cdot) \Rightarrow G(\cdot)$ on 
$\ell_{\infty}(E\times I)$. Then, there exists a stochastic process $\hat G$ on $(\Omega,\mathcal{F}, P)$ with sample paths in $D$, $P(G= \hat G)=1,$ and $W_n(\cdot) \Rightarrow \hat G(\cdot)$ on the metric space $(D,d)$ with metric $d$ given by the sup-norm distance restricted to $D$.
In addition, if $G$ also has sample paths only in $D$, then  
$W_n(\cdot) \Rightarrow G(\cdot)$ on $(D,d).$ 
\end{lem}

\begin{proof}
Since $W_n(\cdot) \Rightarrow G(\cdot)$ on 
$\ell_{\infty}(E\times I)$, by definition of convergence in law in this setting $G$ is a Borel measurable mapping from $(\Omega,\mathcal{F})$ into $(\ell_{\infty}(E\times I), \mathcal{B})$, where $\mathcal{B}$ denotes the Borel subsets of  $\ell_{\infty}(E\times I)$. Therefore, Theorem 1.3.4 of \cite{vw} implies 
$$
\limsup_{n \rightarrow \infty}P^{*}(W_n(\cdot) \in D) \le P( G \in D),
$$
and hence $P(G \in D)=1$. 

If we define $\hat G(\cdot, \omega)=G(\cdot,\omega)$ for $\omega \in \{\omega: G^{-1}(D)\}$, and to be $d_0$ for some $d_0 \in D$ on $ \{\omega: G^{-1}(D^c) \}$, then $\hat  G$ is also Borel measurable, $P(G=\hat G)=1$, $W_n(\cdot) \Rightarrow \hat G(\cdot)$ on 
$\ell_{\infty}(E\times I)$, and by Theorem 1.3.10 of \cite{vw} we then also have $W_n(\cdot) \Rightarrow \hat G(\cdot)$ on $D$.

Clearly $\hat G$ is Borel measurable, $P(G= \hat G)=1$, and $W_n(\cdot) \Rightarrow  G(\cdot)$ on 
$\ell_{\infty}(E\times I)$. Moreover, if $G$ only takes values in $D$, then $G =\hat G$ on all of $\Omega$.  

\end{proof}

\subsection{Proof of Theorem \ref{quant-CLT-paths}}

\begin{proof}

If $\{X_t: t \in E\}$ has a version with paths in $\mathbb{C}_{e_1}(E)$, then taking i.i.d. copies of this continuous version to build the quantile processes, the proof of Lemma 3 implies one has with probability one that $W_n(t,\alpha)$ is continuous on $( E,e_1)$ for each $\alpha \in (0,1)$. In addition, for each $n \ge 1$, by Lemma 3 we have $\alpha \rightarrow \tau_{\alpha}^n(t)$ is in $ \mathbb{D}_2(I)$ for all $t\in E$, and therefore for all $n \ge 1$ the paths of $W_n$ satisfy
\begin{equation}\label{quant-CLT-functions-6}
W_n(\cdot,\cdot) \in \mathbb{C}_{e_1}(E) \otimes\mathbb{D}_2(I).
\end{equation}
Moreover, since we are assuming
the empirical quantile CLT with $W_n \Rightarrow \tilde G$ on $\ell_{\infty}(E \times I)$, Lemma \ref{weak-conv-paths} implies there is a Gaussian process $\hat G$ with values in $\mathbb{C}_{e_1}(E) \otimes\mathbb{D}_2(I)$ such that $\hat G$ and $\tilde G$ induce the same Borel probability measure on $\ell_{\infty}(E\times I)$ and $W_n \Rightarrow \hat G$ on $\mathbb{C}_{e_1}(E) \otimes\mathbb{D}_2(I)$ with the topology that is given by the sup-norm. Thus (i) holds.

The proof of (ii) is entirely similar, since the assumptions of (ii) and Lemma \ref{ctble-unctble-sups} imply that (\ref{quant-CLT-functions-6}) holds with $\mathbb{C}_{e_1}(E) \otimes\mathbb{D}_2(I)$ replaced by $\mathbb{D}_1([0,T])\otimes\mathbb{D}_2(I)$. 
Hence (ii) is verified as before.
\end{proof}

\subsection{Applications of Theorem \ref{quant-CLT-paths}} 
Here we indicate some self-similar processes $X$ to which Theorem \ref{quant-CLT-paths} applies, and provides a CLT for its empirical quantiles in classical function spaces. Immediate examples involve the processes studied in Theorem \ref{scale-quant-app}, but there are many other self-similar processes to which the conclusions in Theorem \ref{scale-quant-app} also apply. Moreover, it is important to note that Theorem \ref{scale-quant-app} depends heavily on Theorem \ref{scale-quant}, so we are able to obtain empirical quantile CLTs under circumstances not available from \cite{kz-quant} alone, but Theorem \ref{quant-CLT-paths} can be applied  to obtain results in that setting as well.

\subsubsection{Symmetric Stable Processes}

The parameterizations in $r$ are as before, and our first application assumes the input process $X$ is a cadlag symmetric $r$-stable process with stationary independent increments, $P(X(0)=0)=1$, and $0 <r<2.$ Then, for $T \in (0,\infty)$
and $I$ a closed subinterval of $(0,1)$ we have the empirical quantile CLT of Theorem \ref{scale-quant-app} in $\ell_{\infty}([0,T]\times I)$.
Hence the conclusions of part (ii) of Theorem \ref{quant-CLT-paths} apply to these empirical quantile processes. In particular, for fixed $\alpha \in (0,1)$ the empirical quantile CLT holds in the Banach space $\mathbb{D}_1([0,T])$ with the sup-norm.

\subsubsection{Fractional Brownian Motions}

Let $X=\{X(t): t \ge 0\}$ be a centered sample continuous $r$-fractional Brownian motion, where $0<r<2$, $X(0)=0$ with probability one, and $\E(X_t^2)= t^{r}$ for $t \ge 0$.
Then, Theorem \ref{scale-quant-app} implies the empirical quantile CLT for these processes in $\ell_{\infty}([0,T]\times I)$, and hence the conclusions of part (i) of Theorem \ref{quant-CLT-paths} with $E=[0,T]$ apply to the empirical quantile processes with these sample continuous inputs. The special case $r=1$ implies $X$ is Brownian motion, and for fixed $\alpha \in (0,1),$ the quantile CLT holds in the Banach space $\mathbb{C}_{e_1}([0,T])$ with the sup-norm. As mentioned earlier, this implies the CLT for medians in  \cite{swanson-scaled-median}, and for other individual quantile levels $\alpha \in (0,1)$ in \cite{swanson-fluctuations}, even if $P(X(0)=0)=1$. Moreover, if $\alpha=\beta=\frac{1}{2}$ and $s,t \in (0,T]$, the covariance of the limiting Gaussian process for $0<r<2$ is
$$
\frac{P(X(s) \le 0,X(t) \le 0)- \frac{1}{4}}{f(s,0)f(t,0)}= s^{\frac{r}{2}} t^{\frac{r}{2}} \sin^{-1}(\frac{E(X_sX_t)}{s^{\frac{r}{2}}t^{\frac{r}{2}}})
$$
where $2E(X_sX_t)= s ^{r} + t^{r} - |s -t|^{r}$, and the equality follows from a standard Gaussian identity. Of course, when $r=1$ 
this is as in \cite{swanson-scaled-median}.

\subsubsection{Integrated Brownian Motion and Self-Similar Iterated Processes}

The conclusions in part (i) of Theorem \ref{quant-CLT-paths} also hold when $\{X(t): t \ge 0\}$ is an $m$-times integrated  sample continuous Brownian motion, and also for a variety of iterated processes of the form $\{ B(r(t)): t \ge 0\}$, where
$\{B(t): t \in \mathbb{R}\}$ is a centered sample continuous Brownian motion on $\mathbb{R}$ such that $E(B(s)B(t)) = \min(|s|,|t|)$ for $s$ and $t$ of the same sign, and zero otherwise. The process $\{r(t): t \ge 0\}$ is assumed to be self-similar, sample continuous, independent of the Brownian motion, and sufficiently general that these examples cover iterated Brownian motion, L\'evy's stochastic area process, and also many other situations. 

The general approach to the proof of these results is as indicated for the previous examples. That is, as in Theorem \ref{scale-quant-app} one shows that conditions (A-1),$\cdots$,(A-5) hold, and hence Theorem \ref{scale-quant} implies the empirical quantile CLT in $\ell_{\infty}(E \times I)$. Then Theorem \ref{quant-CLT-paths} is applicable to obtain the quantile CLTs in spaces with more regular sample paths. Except for verifying 
the condition (A-3) for the iterated processes, checking these conditions is similar to what is done above. Hence the details are not included here, but will be posted in an appendix to our manuscript online.

\begin{rem}
The Brownian sheet has continuous paths and is close enough to being self-similar that by using the methods above, results of the type discussed here should also hold for the sheet. We have checked these results when the input process is the 2-parameter sheet on $[0,T]\times[0,T]$, but the broad scope of the argument is similar to what is done in this paper, so they also are not included. The Brownian sheet is of particular interest in this setting since Corollary 2 of \cite{kz-quant} shows the empirical process CLT fails for the tied down sheet on $[0,T] \times [0,T]$, and hence the self-similarity properties are required to get the conclusions needed for the empirical quantile process CLT in $\ell_{\infty}([0,T] \times[0,T])$.

\end{rem} 

\section{Additional Comments on Lemma \ref{ctble-unctble-sups}}

The referee has pointed out that there is a stronger version of Lemma \ref{ctble-unctble-sups}, which, while not needed for the processes we consider, could eventually apply to a wider class of processes. We state this refinement in Theorem A below. The lemma following Theorem A uses the more global assumption of H-self-similarity to easily imply an extension of Lemma \ref{ctble-unctble-sups}. A useful reference for a proof of Theorem A is the material of section \ref{sec:5}, and also some of the results referenced there. 

Previously, we dealt with quantiles as defined in (\ref{quantile}). These may be called the minimal quantiles, since they give the left-most quantile. The new lemma also deals with several new quantities including maximal quantiles and the processes, $\{X(s-): s>0\}$ and the corresponding iid copies. Namely, we consider 
\begin{align*}
H(t,x)=P(X(t-) \le x) \ &\text{ and for all }  t > 0,\ x \in \mathbb{R}\ \text{ and }\  \alpha \in (0,1),\\ 
H^{-1}(t,\alpha)&= \inf\{x:H(t,x) \ge \alpha\},
\end{align*}
and for all $\omega \in \Omega$ and $n \ge 1$ let
\[
H_n(t,x,\omega):= \frac{1}{n} \sum_{i=1}^n I_{(-\infty, x]}(X_i(t-,\omega)),\forall t > 0,\forall x \in \mathbb{R},
\]
as well as the associated minimal quantiles 
\[H^{-1}(t,\alpha)= \inf\{x:H(t,x) \ge \alpha\}\, \text{ and } H_{n}^{-1}(t,\alpha)=\inf\{x: H_{n}(t,x)\ge \alpha\}.
\]

 \begin{thma}\label{thma} Let $\alpha,\alpha_1,\alpha_2,\ldots \in (0,1)$ be given numbers such that $\alpha_k \rightarrow \alpha$. Let $s,s_1,s_2,\ldots$ be given numbers such that $0 \le s_k<s$ and $s_k \uparrow s$. Let
$t,t_1,t_2,\ldots$ be given numbers such that $0 \le t< t_k$ and $t_k \downarrow t$. If $\omega \in \Omega$ and $n \ge 1 $ we have:

\begin{enumerate} \item If $X(t)$ admits a unique $\alpha$-quantile, then 
\[ 
F^{-1}(t,\alpha)= \lim_{k \rightarrow \infty}  F^{-1}(t_k,\alpha_k).
\]
\item If $X(s-)$ admits a unique $\alpha$-quantile, then 
\[ 
H^{-1}(s,\alpha)= \lim_{k \rightarrow \infty}  F^{-1}(s_k,\alpha_k).
\]
\item If either $\alpha \notin \{\frac{1}{n},\ldots,\frac{n-1}{n}\}$ or $\sup_{k \ge 1}\alpha_k \le \alpha$, we have
\end{enumerate}
\begin{align*}&(a)\ \ F_n^{-1}(t,\alpha,\omega) = \lim_{k \rightarrow \infty}F_n^{-1}(t_k,\alpha_k,\omega),\\
&(b)\ \ H_n^{-1}(s,\alpha,\omega) = \lim_{k \rightarrow \infty}F_n^{-1}(s_k,\alpha_k,\omega).%~\rm{ when}~ s=t,
\end{align*}

\noindent Moreover, if $X(t)$ and $X(s-)$ have unique $\alpha$-quantiles for all $t \ge 0$ and all $s>0$, then $ F^{-1}(\cdot,\alpha)$ and  $ F_n^{-1}(\cdot,\alpha,\omega)$ are cadlag functions, and:

\begin{enumerate} 
\item[(4)] $ F^{-1}(s-,\alpha)=  H^{-1}(s,\alpha), ~  F_n^{-1}(s-,\alpha,\omega)=  H_n^{-1}(s,\alpha,\omega) \forall s>0.$

\item[(5)] If $s>0$ and $X(s)=X(s-)$ in distribution, then $ F^{-1}(\cdot,\alpha)$ is continuous at $s$.
\end{enumerate}
\end{thma}

Combining the following simple lemma with Theorem A, we obtain the results in Lemma \ref{ctble-unctble-sups}, and also additional results involving  joint behavior in $t$ and $\alpha$ for the quantiles and empirical quantiles.

\begin{lema}\label{lema} Let $\{X(t): t \ge 0\}$ be cadlag and $H$-self-similar for some $H>0$, and let $\alpha \in (0,1)$ and $u>0$ be given numbers such that $X(u)$ has a unique $\alpha$-quantile. Then, $X(t)$ has a unique $\alpha$-quantile for all $t \ge 0,$ and for all $s>0$ we have $X(s-)=X(s)$ in distribution and $X(s-)$ has a unique $\alpha$-quantile.
\end{lema}

{\bf Acknowledgement:} It is a pleasure to thank the referee for a careful reading of the manuscript. His comments and suggestions led to a number of refinements, and in particular improved our previous statement of Theorem \ref{quant-CLT-paths}. We also thank the referee for allowing us to include Theorem A and Lemma A in the paper.

\end{document}